\documentclass{amsart}

\usepackage[colorlinks=true,urlcolor=blue,linkcolor=blue,citecolor=black]{hyperref}
\usepackage{amssymb,amsmath, amscd, mathrsfs, yfonts, bbm, graphics, dsfont, footmisc }
\usepackage{xcolor}
\usepackage{tikz}
\usepackage{float}
\usetikzlibrary{calc, arrows.meta,decorations.pathreplacing}

\usepackage{caption}
\usepackage{bbm}
\usetikzlibrary{positioning,shapes,backgrounds}

\title[] {On some properties of free commutators with semicircular variables}

\author[M. Popa]{Mihai Popa}
\address[M. Popa]{Department of Mathematics\\ University of Texas at San Antonio\\ One UTSA Circle San Antonio\\ Texas 78249, USA
	\\ and “Simon Stoilow” Institute of Mathematics of the Romanian Academy\\ P.O. Box 1-764\\ 014700 Bucharest, Romania}
\email{mihai.popa@utsa.edu}

\author[K. Szpojankowski]{Kamil Szpojankowski}
\address[K. Szpojankowski]{Institute of Mathematics of the Polish Academy of Sciences, ul. \'Sniadeckich 8, 00–656 Warszawa, Poland\\and
	Wydzia\l{} Matematyki i Nauk Informacyjnych\\
	Politechnika Warszawska\\
	ul. Koszykowa 75\\
	00-662 Warszawa, Poland.}
\email{kamil.szpojankowski@pw.edu.pl}
\thanks{KSz: This research was funded in part by National Science Centre, Poland WEAVE-UNISONO grant BOOMER 2022/04/Y/ST1/00008.
}

\usepackage{graphicx}

\newtheorem{claim}{}[section]

\newtheorem{notation}[claim]{Notation}
\newtheorem{thm}[claim]{Theorem}
\newtheorem{lemma}[claim]{Lemma}

\newtheorem{prop}[claim]{Proposition}
\newtheorem{cor}[claim]{Corollary}

\newcommand{\bE}{\mathbb{E}}

\newcommand{\cA}{\mathcal{A}}

\input xy
\xyoption{all}

\usepackage{graphicx}

\begin{document}

\begin{abstract}
We investigate commutators of free variables of the form \( i[x, s] \), where \( s \) is a semicircular element. We show that although \( s \) and \( i[x, s] \) are not free, their sum nevertheless satisfies the free additive convolution identity
\[
\mu_{s + i[x, s]} = \mu_s \boxplus \mu_{i[x, s]}.
\]
Furthermore, we prove that the polynomial \( x + i[x, s] \) is freely infinitely divisible whenever \( x \) itself is freely infinitely divisible.

\end{abstract}

\maketitle


\section{Introduction}

The semicircular distribution is often seen as the Free Probability analogue of the Gaussian distribution (see, for example \cite{biane}, \cite{mingo-speicher}, \cite{nisp}).
Still, there are many properties of each of these distributions that do not find analogues for the other (such as in \cite{eisembaum}). In the classical setting any commutator is trivial, which is obviously not the case in the free setting. 
The  paper presents some properties of the commutator $ i[x, s] $, where $s $ is a semicircular random variable and $ x $ is a random variable free from $ s $.
 First, it is well known that equality between the distribution of \( X + Y \) and the convolution of the marginal distributions of \( X \) and \( Y \) does not, in general, imply that \( X \) and \( Y \) are independent.  
Independence requires that the joint characteristic function factorizes; that is, for any \( s,t \in \mathbb{R} \),
\[
\varphi_{(X,Y)}(s,t) = \varphi_X(s)\,\varphi_Y(t).
\]
In contrast, the convolution property only requires that for any \( s \in \mathbb{R} \),
\[
\varphi_{X+Y}(s) = \varphi_{(X,Y)}(s,s) = \varphi_X(s)\,\varphi_Y(s).
\]

A similar phenomenon occurs in the context of \emph{freeness} and \emph{free convolution}.  
Namely, the fact that a sum \( a + b \) has distribution equal to the free convolution of the marginal distributions, \( \mu_a \boxplus \mu_b \), does not necessarily imply that \( a \) and \( b \) are freely independent.  
While the classical analogue of this observation has been extensively studied and illustrated by numerous examples in the literature (see, for example, \cite{hamedani}), \cite{schennach}), only a few analogous examples are known in the framework of free probability.  
In this paper, we provide a broad new class of such examples.

In~\cite{LehnerSzpojankowski}, the authors showed - using subordination techniques - that if \( s \) is a semicircular element and \( x \) is a Bernoulli random variable free from \( s \), then the element \( s + i[s,x] \) has distribution
\[
\mu_{s + i[s,x]} = \mu_s \boxplus \mu_{i[s,x]},
\]
yet \( s \) and \( i[s,x] \) are not free, since their fourth joint free cumulant does not vanish.  
We extend this result to arbitrary distributions of \( x \).  
Our approach differs from that of~\cite{LehnerSzpojankowski}: we develop a combinatorial description of the distribution of \( s + i[s,x] \).  
Thus we show that the following theorem.

\begin{thm}
	\label{thm:main1}
	Let \( s \) and \( x \) be free random variables such that \( s \) is semicircular. Then the distribution of \( s + i[s, x] \) is the free additive convolution of the distributions of \( s \) and \( i[s, x] \). Moreover, the elements \( s \) and \( i[s, x] \) are not free.
\end{thm}

This result yields new insight about \emph{free infinite divisibility} (FID).  
Arizmendi, Hasebe and Sakuma~\cite{AHS} showed that commutators preserve free infinite divisibility: if \( x \) and \( y \) are free random variables whose distributions are FID, then the distribution of \( i[x,y] \) is also FID. Free infinite divisibility of commutators was studied further by Ejsmont and Lehner \cite{LehnerEjsmont}

An immediate consequence of Theorem~\ref{thm:main1} is that if \(x\) and \(s\) are free and \( x \) is FID, then \( i[x,s] \) is FID as well, and hence so is \( s + i[x,s] \).  
This observation naturally raises the question of whether, under the same assumptions, the element \( x + i[x,s] \) is also FID. Our answer is affirmative, although in this case the distribution of \( x + i[x,s] \) no longer equals \( \mu_x \boxplus \mu_{i[x,s]} \).

The proof proceeds in two steps.  
First, we derive a combinatorial description of the free cumulants of \( x + i[x,s] \).  
To formulate this result, we introduce the following notation.  
Fix an interval partition on \( n \) elements, \( \sigma \in \mathrm{Int}(n) \), and let \( |\sigma| = k \) denote the number of blocks of \( \sigma \).  
For such a partition \( \sigma \), consider a non-crossing partition \( \pi \in \mathrm{NC}(k) \).  
We denote by \( \pi(\sigma) \) the non-crossing partition \( \rho \in \mathrm{NC}(n) \) whose blocks are unions of the blocks of \( \sigma \) corresponding to indices in the same block of \( \pi \).  
More precisely, if \( \sigma = \{W_1, \ldots, W_k\} \) and \( \pi = \{V_1, \ldots, V_m\} \), then
\[
\rho = \pi(\sigma) = \{Z_1, \ldots, Z_m\}, \quad \text{where } Z_i = \bigcup_{j \in V_i} W_j, \; i = 1,\ldots,m.
\]

This construction is illustrated in the following example.
\begin{figure}
\begin{tikzpicture}[scale=1.0, every node/.style={font=\small}, xshift=-3cm]
	
	\node at (-1.0,2.5) {\textbf{Left: $\sigma$ and $\pi$}};
	\node at (6,2.5) {\textbf{Right: $\rho = \pi(\sigma)$}};
	
	\foreach \i in {1,...,8}
	\node[dotted] (L\i) at (0.6*\i-3.5,0) {\i};
	\draw[thick,blue] (L1.north) |- ($(L2.north)+(0,0.5)$)|-(L2.north);
	\draw[thick,blue] (L3.north) |- ($(L4.north)+(0,0.5)$)|-(L4.north);
	\draw[thick,blue] (L5.north) |- ($(L6.north)+(0,0.5)$)|-(L6.north) |-($(L6.north)+(0,0.5)$)|-($(L7.north)+(0,0.5)$)|-(L7.north);
	\draw[thick,blue] ($(L8.north)+(0,0.5)$)|-(L8.north);	
	\node[below=8pt of L1.south , anchor=west] 
	{$\sigma=\{\{1,2\},\{3,4\},\{5,6,7\},\{8\}\}$};
	
	\coordinate (B1) at (-2.6,1);
	\coordinate (B2) at (-1.4,1);
	\coordinate (B3) at (0.1,1);
	\coordinate (B4) at (1.3,1);
	
	\draw[thick,red](B2.north)|- ($(B4.north)+(0,0.75)$)|-(B4.north);
	\draw[thick,red] ($(B3.north)+(0,0.5)$)|-(B3.north);
	\draw[thick,red] ($(B1.north)+(0,0.5)$)|-(B1.north);
	
	\node at (-0.5,2) {$\pi=\{\{1\},\{2,4\},\{3\}\}$};
	
	\begin{scope}[shift={(6,0)}]
		\foreach \i in {1,...,8}
		\node[dotted] (R\i) at (0.6*\i-3.5,0) {\i};

		\draw[thick,red] (R3.north) |- ($(R4.north)+(0,0.5)$)|-(R4.north)|- ($(R8.north)+(0,0.5)$)|-(R8.north);
		
		\draw[thick,red] (R1.north) |- ($(R2.north)+(0,0.5)$)|-(R2.north);
		\draw[thick,red] (R5.north) |- ($(R6.north)+(0,0.3)$)|-(R6.north)|- ($(R7.north)+(0,0.3)$)|-(R7.north);

		\node[below=8pt of R1.south east, anchor=west] 
		{$\rho=\pi(\sigma)=\{\{1,2\},\{3,4,8\},\{5,6,7\}\}$};
	\end{scope}
\end{tikzpicture}
\caption{Example of partition $\sigma\in \mathrm(8)$ with 4 blocks together with $\pi\in NC(4)$ and corresponding partition $\pi(\sigma)\in NC(8)$}
\label{fig:1}
\end{figure}
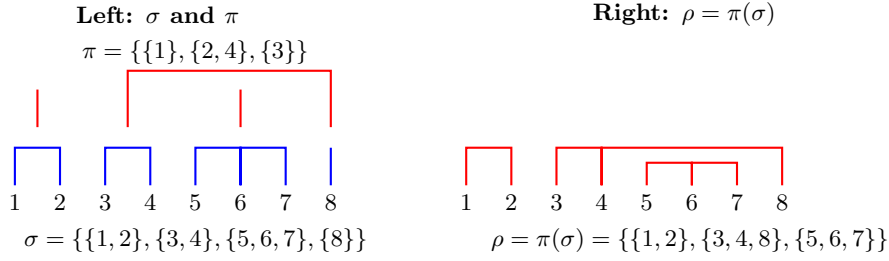

We denote by \( \mathrm{Int}_{\geq 2}(n) \) the set of interval partitions of \( \{1,\ldots,n\} \) in which every block contains at least two elements.  
Let \( \mathrm{NC}(k) \) denote the set of non-crossing partitions of \( \{1,\ldots,k\} \) such that \( 1 \) and \( k \) belong to the same block.  
The description of the free cumulants of \( x + i[x,s] \) is given in the following proposition.

\begin{prop}\label{prop:intro}
	Let \( x, s \) be free random variables, and suppose \( s \) is semicircular. Then
	\[
	\kappa_n(x+i[x,s])
	=\kappa_n(x)+ \sum_{k=1}^{\lfloor n/2 \rfloor} 
	\sum_{\substack{\sigma \in \mathrm{Int}_{\geq 2}(n) \\ |\sigma|=k,\, 1 \in V_1}}
	\sum_{\pi \in \mathrm{NC}(k)} 
	|V_1|\, \kappa_{\pi(\sigma)}(x).
	\]
\end{prop}

Under the additional assumption that \( x \) is a compound free Poisson element, we construct an operator model whose moments coincide with the free cumulants of \( x + i[x,s] \).  
Finally, using the fact that compound free Poisson elements are dense among FID distributions, we obtain our second main result.

\begin{thm}\label{thm:main2}
	Let \( x \) and \( s \) be free, with \( x \) freely infinitely divisible and \( s \) semicircular. Then the distribution of \( x + i[x,s] \) is freely infinitely divisible.
\end{thm}

Besides this introduction this paper contains 4 more sections.  
Section~2 provides the necessary background and fixes notation.  
In Section~3 we prove Theorem~\ref{thm:main1}.  
Section~4 is devoted to the study of the element \( x + i[x,s] \) and contains the proof of Proposition~\ref{prop:intro}.  
Finally, in Section~5 we construct an operator model, in the spirit of \cite{mvpvv}, \cite{mvp-monot}, showing that the free cumulants of \( x + i[x,s] \) correspond to moments of a positive measure, thereby establishing Theorem~\ref{thm:main2}.

\section{Preliminaries}\label{section:1}

For the whole paper we fix a non--commutative probability space $(\cA,\tau)$ such that $\cA$ is a unital $C^*$--algebra and $\tau$ is positive, faithful, tracial linear functional, such that $\tau(1_{\cA})=1$.

\begin{notation}
	If $ \pi \in NC(n)$ and $ A \subseteq [n] $, we shall write
	 $ A \trianglelefteq \pi $ 
	 if $ A $ is a union of blocks of $ \pi $.
\end{notation}

 	Remark that if 
 	$ A \trianglelefteq \pi $, then also
 	$ [n]\setminus A \trianglelefteq \pi$.
 	Furthermore, if $(i_1, i_2, \dots, i_l)$ is a block of $\pi $ with $ i_1 < i_2 < \dots , i_l $, then $ \{ i_1+1, i_1 +2, \dots, i_l \} \trianglelefteq \pi $
 	and
 	$ \{ i_{q}+1, i_{q}+ 2, \dots, i_{q+1}- 1\} \trianglelefteq \pi $
 	for each 
 	$ q \in \{ 1, 2, \dots, l-1\}$.
 	
%
%
%
%
  \begin{notation}\label{notation:23}
 	Let $\{B_1, B_2, \dots, B_l\}$ be a partition of the set $[n] = \{1, 2, \dots, n\}$.
 	For each $i \in [l]$, write
 	\[
 	B_i = \{ b_{i,1} < b_{i,2} < \dots < b_{i,q(i)} \},
 	\]
 	and let $m(1), \dots, m(l)$ be fixed positive integers.
 	$\ $
 	We introduce the shorthand notation
 	\[
 	\big( (X_{1,1}, X_{1,2}, \dots, X_{1,m(1)})_{|B_1}, \dots, (X_{l,1}, X_{l,2}, \dots, X_{l,m(l)})_{|B_l} \big)
 	\]
 	to denote the ordered $n$-tuple
 	\[
 	(x_1, x_2, \dots, x_n),
 	\]
 	where each component $x_j$ is determined as follows:
 	for $j = b_{u,v} \in B_u$, we set
 	\[
 	x_j = X_{u,s} \quad \text{with} \quad s \equiv v \pmod{m(u)}.
 	\]
 	In other words, within each block $B_u$, the variables
 	$X_{u,1}, \dots, X_{u,m(u)}$
 	are assigned cyclically according to the order of indices in $B_u$.
 	$\ $
 	If $m(u) = 1$ for some $u \in [l]$, we adopt the simplified notation
 	\[
 	X_{u,1\,|B_u}
 	\quad \text{instead of} \quad
 	(X_{u,1})_{|B_u}.
 	\]
 \end{notation}

 For example
 \begin{align*}
& \big( (X, Y, Z)_{| \{2, 3, 4, 5\}}, (T, V)_{|\{1, 6, 7 \} } \big) 
= \big( T, X, Y, Z, X, V, T \big) \\
& \big( (X, Y)_{| \{ 1, 2, 3, 4, 7\}}, Z_{|\{5, 6\}} \big) = 
 \big( X, Y, X, Y, Z, Z, X\big).
 \end{align*}
 
%
%
%
 
 \begin{notation}
 Let $A \subset [n] = \{1, 2, \dots, n\}$ with
 $A = \{a_1 < a_2 < \dots < a_k\}$ and $1 \le k < n$.
 
 \smallskip
 \noindent
 (1) The map $\varphi_A$.
 Define a surjective map
 \[
 \varphi_A : [n+k] \to [n]
 \]
 by
 \[
 \varphi_A(s) = s - |\{t \in A : t < \varphi_A(s)\}|,
 \]
 that is, $\varphi_A$ collapses each pair of consecutive integers
 corresponding to elements of $A$ back to a single index in $[n]$. 
 
 \smallskip
 \noindent
 (2) The inverse image map $\iota_A$.
 For any subset $B \subset [n]$, define
 \[
 \iota_A(B) := \varphi_A^{-1}(B) \subset [n+k].
 \]
 Thus, $\iota_A$ “expands’’ a subset of $[n]$ into the corresponding subset of $[n+k]$, duplicating each element in $A$.
 
 \smallskip
 \noindent
 (3) The partition $\tau_A$.
 Let $\tau_A$ denote the interval partition of $[n+k]$ defined as follows:
 for each $j \in [n]$,
 \[
 B_j =
 \begin{cases}
 	\{\, \varphi_A^{-1}(j) \,\} = \{t\}, & \text{if } j \notin A,\\[3pt]
 	\{t, t+1\}, & \text{if } j \in A,
 \end{cases}
 \]
 where $t$ is the smallest element in $\varphi_A^{-1}(j)$.
 In other words, each $j \in A$ corresponds to a block of two consecutive elements, and each $j \notin A$ corresponds to a singleton block:
 \[
 \tau_A = \{B_1, B_2, \dots, B_n\}.
 \]
 
 \smallskip
 \noindent
 Example.
 If $n=5$ and $A = \{2, 4\}$, then
 \[
 \tau_A = \{(1), (2,3), (4), (5,6), (7)\},
 \quad
 \iota_A(\{1,3,4,5\}) = \{1,4,5,6,7\}.
 \]
 \end{notation}

The following result (Theorem 11.12 from \cite{nisp}) will be extensively used in the
following sections.
\begin{thm}\label{thm:product-entries}
Consider the non-commutative probability space $(\mathcal{A}, \tau)$ and let
$(\kappa_\pi)_{\pi \in NC} $ be the correspondent free cumulants. Let $m $ be a positive integer and let $\sigma $ be an interval partition of $[n]$ with $m$ blocks, 
\[ \sigma = \big(1, \dots, i(1)\big), \big(i(1)+1, \dots, i(2)\big), 
\dots, 
\big(i(m-1)+1, \dots, n\big).
\]
For any variables $ a_1, a_2, \dots, a_n \in \mathcal{A}$  we have that
\[ 
\kappa_m \big( a_1\cdots a_{i(1)}, a_{i(1)+1} \cdots a_{i(2)}, \dots, a_{i(m-1)+1} \cdots a_n \big)
= 
\sum_{\substack{\pi\in NC(n)\\ \pi \vee \sigma = \mathbbm{1}_n}}
\kappa_{\pi}\big[a_1,a_2, \dots, a_n\big]
\]

\end{thm}


\section{Free convolution without freeness $s+i[s,x]$}\label{Section:3}

\begin{lemma}\label{lem:31}
 If $ s $ and $ x $ are free random variables such that $s $ is semicircular, then $ s $ and $ i[s, x] $ are \textbf{not} free.
\end{lemma}
 
\begin{proof}

It suffices to show that if 
 is even, then 
\begin{equation}\label{k:even}
\kappa_{4}\big(s, i[s, x], i[s, x], s \big) \neq 0. 
\end{equation} 

By linearity we have
\begin{align*}
	\kappa_{4}\big(s, i[s, x], i[s, x], s \big) =&-\kappa_4(s,sx,sx,s)+\kappa_4(s,sx,xs,s)\\ &+\kappa_4(s,xs,sx,s)-\kappa_4(s,xs,xs,s).
\end{align*} 
Observe that only second term in the sum above is non--zero. Indeed for $\kappa_4(s,sx,sx,s)$ consider which $s$ is paired with the last $s$, and observe that all possible pairing force the partition to violate the maximality property from Theorem \ref{thm:product-entries}. For $\kappa_4(s,sx,sx,s)$ observe that we can pair the first $s$ only with second or last $s$, otherwise we will separate second $s$ and $\kappa_1(s)=0$. But again joining first $s$ with second or last $s$ forces violation of maximality again. Similar argument applies to $\kappa_4(s,xs,xs,s)$.

For the second term similar analysis as above leads to
\begin{align*}
	\kappa_4(s,sx,xs,s)=\kappa_2(s)\kappa_2(x)\kappa_2(s)>0.
\end{align*} 

\end{proof}

\begin{thm}
	If $ s $ and $ x $ are free random variables such that $ s $ is semicircular, then the distribution of $ s + i[s, x] $ is the free additive convolution of the distributions of $ s $ and of $i[s, x]$.
\end{thm}

 \begin{proof}
 
 We shall show that although in view of Lemma \ref{lem:31} the joint cumulants of $s$ and $i[s,x]$ do not vanish in general, for any positive integer $ n $, we still have
 \begin{equation*}
 \kappa_n\big( s + i[s, x] \big) = \kappa_n \big( i[s, x]\big) + \kappa_n \big(s\big).
 \end{equation*}	
Using Notation \ref{notation:23} it suffices to show that for each $ 1 \leq k < n $, 
\begin{equation}
\sum_{\substack{B \subset [n] \\ | B | = k}}
\kappa_n( [s, x]_{| B }, s_{| [n]\setminus B })=0.
\end{equation}
 From linearity of free cumulants the equation above is equivalent to
 \begin{equation}\label{kAB}
 \sum_{\substack{B \subset [n] \\ | B | = k}} 
 \sum_{D \subseteq B } 
 (-1)^{|D|}\kappa_n( sx_{| B \setminus D}, xs_{|D}, s_{| [n]\setminus B  })=0
 \end{equation}	
 which, from Theorem \ref{thm:product-entries}, reads
 \begin{equation}\label{eq:3s}
  \sum_{\substack{B \subset [n] \\ | B | = k}} 
 \sum_{D \subseteq B }
 \sum_{\substack{\pi \in NC(n+k) \\ \pi \vee \tau_B = \mathbbm{1}_{n+k} } } 
 (-1)^{| D|}
 \kappa_{\pi} \big[ 
 (s, x)_{ | \iota_B (B \setminus D)} , (x, s)_{| \iota_B (D)} , s_{| \iota_B ([n] \setminus B )}\big] = 0.
 \end{equation}
 
 Let 
 $\big(x_1, x_2, \dots, x_{n+k}\big) = 
 \big(
 (s, x)_{ | \iota_B (B \setminus D)} , (x, s)_{| \iota_B (D)} , s_{| \iota_B ([n] \setminus B )}\big) $

 Remark that either the summand in (\ref{eq:3s}) is zero or 
 that  $ \pi \vee \tau_B \neq \mathbbm{1}_{n+k}  $
unless $ \pi $ satisfies the following conditions:
\begin{enumerate}
	\item[(i)] \textit{If $ i $ and $j $ are in the same block of $ \pi $, then $ x_i = x_j $; furthermore, if $ x_i =s $, then $ (i, j) $ is a block of $\pi $}.
	\item[(ii)] \emph{If $ i $ and $ j $ are in the same block of $ \pi $ with $ i \in \iota_B([n]\setminus B ) $ 
	then $ j \in \iota_B(B)$, $ x_i = x_j= s $ and 
	$ i, j$ are consecutive mod. $n+k$ (i.e. $\{i, j\} = \{ n+k, 1\} $ or $ \{i, j\} \in \big\{ \{ t, t+1\}: \ 1 \leq t < n+k \big\} $)}.
\end{enumerate}

  Part (i) follows trivially from the free independence of $ s $ and $ x $ and from the properties of the semicircular distribution.
  
   For part (ii) suppose first that if 
  $ i, j \in \iota_B([n]\setminus B ) $
   are in the same block of
    $\pi $; 
    then
     $x_i = x_j = s $.
      Since $ s $ is semicircular,
       $ (i, j)$ 
       must be a block of $\pi$. But
  $ (i), (j)$
   are blocks of $ \tau_B $ so
   $(i, j)$ 
   is a block of 
   $ \pi \vee \tau_B $,
    hence
   $ \pi \vee \tau_B \neq \mathbbm{1}_{n+k}  $.
   
    Next, suppose that $ i, j $ are not consecutive, $ i \in \iota_B ([n]\setminus B ) $, $ j \in \iota_B (B) $ and $ i, j $ are in the same block of $ \pi $. 
    From (i), we have that $ x_i=x_j = s $ and $ (i, j) $ is a block of $ \pi $. 
    If $ i < j $ and 
    $ j \in \iota_B( D )$, then $(i)$ and $(j-1, j)$ are blocks of $ \tau_B $ and $ x_{j-1} = x $ (see the Figure \ref{fig:2}).
    \begin{center}
\begin{figure}[H]
	\begin{tikzpicture}[
	scale=1.2,
	baseline=(current bounding box.center),
	every node/.style={font=\small},
	var/.style={
		draw,
		circle,
		minimum size=4mm,
		inner sep=1pt,
		thick
	}
	]
	
	\coordinate (i) at (0,0);
	\coordinate (d1) at (0.8,0);
	\coordinate (d2) at (1.6,0);
	\coordinate (y) at (2.4,0);
	\coordinate (j) at (3.2,0);
	
	\node[var] (Si) at (i) {$s$};
	\node[var] (Y) at (y) {$x$};
	\node[var] (Sj) at (j) {$s$};
	
	\node[below=4pt] at (Si.south) {$i$};
	\node[below=4pt] at (Y.south) {$j\!-\!1$};
	\node[below=4pt] at (Sj.south) {$j$};
	
	\node at (1.2,0) {$\cdots$};
	
	\draw[thick] (Si.north) |- ($(Sj.north)+(0,0.8)$) -| (Sj.north);
	
\end{tikzpicture}
\caption{Configuration in which singleton $s$ on position $i$ is paired with $s$ on position $j$ coming from a pair $xs$.}
\label{fig:2}
\end{figure}
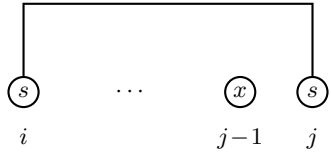

     \end{center}

    
    Since $ \pi $ is non-crossing, 
    $ \{i, i+1, \dots, j\} \trianglelefteq \pi$, and, since $i $ and $j $ are not consecutive mod. $n+k$, 
    $\{i, i+1, \dots, j\} \neq [n+k]$.  On the other hand, since $ \tau_B $ is an interval partition and $(i)$, $(j-1, j)$ are blocks of $ \tau_B $, we have that also
    $ \{i, i+1, \dots, j\} \trianglelefteq \tau_B$
    so the whole interval $ \{i, i+1, \dots, j\}$ is separated in $\pi\vee\tau_B$ and thus   $ \pi\vee \tau_B \neq \mathbbm{1}_{n+k}$. 
    The argument for the case $ i > j $ and 
    $ j \in \iota_B ( B \setminus D)$ is similar (just a mirror reflection of the diagram from figure  above).
    
    If $ i < j $ and $ j \in \iota_B (B \setminus D)$, then $(i)$ and $(j, j+1)$ are blocks of $ \tau_B $ and $ x_{j+1} = x $; furthermore, 
    $ \{ i+1, i+2, \dots, j-1\} \neq \emptyset$ (see the Figure \ref{fig:3}).
    \begin{center}
    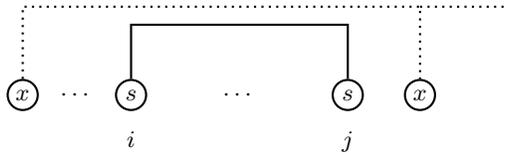
\begin{figure}[H]
    \begin{tikzpicture}[
    	scale=1.2,
    	baseline=(current bounding box.center),
    	every node/.style={font=\small},
    	var/.style={
    		draw,
    		circle,
    		minimum size=4mm,
    		inner sep=1pt,
    		thick
    	}
    	]
    	
    	\coordinate (y1) at (0,0);
    	\coordinate (s1) at (1.2,0);
    	\coordinate (d1) at (2,0);
    	\coordinate (d2) at (2.8,0);
    	\coordinate (s2) at (3.6,0);
    	\coordinate (y2) at (4.4,0);
    	
    	\node[var] (Y1) at (y1) {$x$};
    	\node[var] (S1) at (s1) {$s$};
    	\node[var] (S2) at (s2) {$s$};
    	\node[var] (Y2) at (y2) {$x$};
   
    	\node[below=4pt] at (S1.south) {$i$};
    	\node[below=4pt] at (S2.south) {$j$};
    	
    	\node at (2.4,0) {$\cdots$};
    	\node at (0.6,0) {$\cdots$};
    	\draw[thick] (S1.north) |- ($(S1.north)+(0,0.6)$) -| (S2.north);
    	\draw[thick,dotted] (Y1.north) |- ($(Y1.north)+(0,0.8)$) -| (Y2.north);
    		\draw[thick, dotted] 
    	($(Y2.north)+(0,0.8)$) -- ++(1.0,0);
    \end{tikzpicture}
    \caption{Configuration in which singleton $s$ on position $i$ is paired with $s$ on position $j$ coming from a pair $sx$ with $i<j$.}
    \label{fig:3}
    \end{figure}
\end{center}
    
    
    Since $ (i, j) $ is a block of $\pi $, we have that 
     $ \{ i+1, i+2, \dots, j-1\} \trianglelefteq \pi$. 
     On the other hand, using that $ \tau_B $ is an interval partition and $ (i)$, $(j, j+1)$ are blocks of $ \tau_B $, it follows that 
     $\{ i+1, i+2, \dots, j-1\} \trianglelefteq \tau_B$, so
     again the set $\{ i+1, i+2, \dots, j-1\}$ stays separated and we have $ \pi\vee \tau_B \neq \mathbbm{1}_{n+k}$. Finally, as above, the argument for the case 
     $ i > j $ and $ j \in \iota_B( D)$
     is similar (mirror reflection of the diagram ).
     
     This ends the proof that (ii) must be satisfied. Thus $s$ coming from singletons must be adjacent to $s$ coming from $sx$ or $xs$ and $\pi$ connects these two adjacent elements $s$.  Next we will show that partitions satisfying (i) and (ii) can be decomposed into pairs of partitions which cancel each other out. The idea of behind the rest of the proof is that in partition satisfying (i) and (ii) we can always replace the first pair $(s,sx)$ by $(xs,s)$ by adjusting the blocks containing those variables and keeping the same structure of the rest of the partition. This keeps all free cumulants unchanged while reversing the sign, as $sx$ comes with a plus sign and $xs$ comes with a minus sign.
     
     To simplify the notation, consider the set
     \begin{align*}
     \mathcal{A} = \big\{  (\pi, B, & D):\ \pi \in NC(n+k), \emptyset \neq D \subseteq B \subsetneq [n] \textrm{ such that } | B | = k,\\
     & \pi \vee\tau_B = \mathbbm{1}_{n+k} 
     \textrm{ and } \kappa_{\pi} \big[ 
     (s, x)_{ | \iota_B (B \setminus D)} , (x, s)_{| \iota_B (D)} , s_{| \iota_B ([n] \setminus B )}\big]
      \neq 0\big\}
     \end{align*}
 and, for $(\pi, B, D) \in \mathcal{A} $ we shall use the shorthand notation 
 \begin{align*}
 E(\pi, B, D)= (-1)^{| D|} \kappa_{\pi} \big[ 
 (s, x)_{ | \iota_B (B \setminus D)} , (x, s)_{| \iota_B (D)} , s_{| \iota_B ([n] \setminus B )}\big].
 \end{align*}
 
  We shall write the set $ \mathcal{A} $ as a disjoint union 
  $ \mathcal{A} = \mathcal{A}_1 \sqcup \mathcal{A}_2 $ 
  such that there exists a  bijection $\Phi:\mathcal{A}_1\rightarrow \mathcal{A}_2 $
  with the property that
  $ E \circ \Phi = - E  $.
  
  First, define $ \mathcal{A}_1 $ and $\mathcal{A}_2$ as follows. 
  Let $ (\pi, B, D) \in \mathcal{A} $; according to properties (i) and (ii), the blocks of $\pi $ that contain elements from $ \iota_B([n]\setminus B )$
  are  either $(1, n+k)$ or pairs of consecutive elements. 
  If $ (1, n+k)$ is such a block of $\pi $ then put 
   $ (\pi, B, D) \in \mathcal{A}_1 $ 
   if 
   $ 1 \in \iota_B([n]\setminus B )$,
   respectively put
    $ (\pi, B, D) \in \mathcal{A}_2 $ 
    if
   $ n+k \in \iota_B([n]\setminus B )$.
  If $\pi $ has only such blocks of consecutive elements, let $(i, i+1) $ be the first one in lexicographic order and put
   $ (\pi, B, D) \in \mathcal{A}_1 $ 
   if
    $ i \in \iota_B([n]\setminus B)$, 
    respectively put
   $ (\pi, B, D) \in \mathcal{A}_2 $
   if $ i+1 \in \iota_B([n]\setminus B)$.

  Next, we shall construct the bijection $\Phi$. Suppose first that 
  $(\pi, B, D)\in \mathcal{A}_1$
  is such that
  $(1, n+k)$ is a block of 
  $ \pi$
  and 
  $1 \in \iota_B([n]\setminus B)$. Property (ii) gives then that $ x_1 = x_{n+k}= s $ and that $ n+k \in \iota_B(B) $,
  therefore $ n \in B $ and $ x_{n+k-1}= x $, so $ n \in D $ (see Figure \ref{fig:4}).
  \begin{center}
  	{
  		\setlength{\intextsep}{-2pt}
\begin{figure}[H]
  \begin{tikzpicture}[
  	scale=0.9,
  	baseline=(current bounding box.center),
  	every node/.style={font=\small},
  	var/.style={
  		draw,
  		circle,
  		minimum size=5mm,
  		inner sep=1pt,
  		thick
  	},
  	>=Stealth
  	]
  	
  	\begin{scope}[xshift=0cm]
  		\coordinate (1) at (0,0);
  		\coordinate (y1) at (1.0,0);
  		\coordinate (y2) at (2.0,0);
  		\coordinate (dots) at (3.0,0);
  		\coordinate (yn) at (4.0,0);  
  		\coordinate (nk) at (5.0,0);
  		
  		\node[var] (S1) at (1) {$s$};
  		\node[var] (Y1) at (y1) {$x_1$};
  		\node[var] (Y2) at (y2) {$x_2$};
  		\node[var] (Y) at (yn) {$x$};
  		\node[var] (SNK) at (nk) {$s$};
  		
  		\node at (3.3,0) {$\cdots$};
  		
  		\node[below=4pt] at (S1.south) {$1$};
  		\node[below=4pt] at (Y.south) {$n{+}k{-}1$};
  		\node[below=4pt] at (SNK.south) {$n{+}k$};
  		
  		\draw[thick] (S1.north) |- ($(SNK.north)+(0,1.0)$) -| (SNK.north);
  		
  		\draw[thick,dotted] (Y1.north) |- ($(Y.north)+(0,0.6)$) -| (Y.north);
  		\draw[thick, dotted] (Y2.north) |- ($(Y2.north)+(0,0.6)$);
  		\draw[decorate, decoration={brace, amplitude=10pt, mirror}]
  		($(y1.south west)-(0.3,0.3)$) -- ($(dots.south east)+ (+0.5,-0.3)$)
  		node[midway, below=9pt] {$\rho$};
  	  	\end{scope}

  	\draw [-{Stealth[length=3mm, width=2mm]}] (6,0.5) -- (7,0.5);

  	\begin{scope}[xshift=8cm]
  		\coordinate (1r) at (0,0);
  		\coordinate (yr) at (1.0,0);   
  		\coordinate (y1r) at (2.0,0);
  		\coordinate (y2r) at (3.0,0);
  		\coordinate (dotsr) at (4.0,0);
  		\coordinate (snkr) at (5.0,0);
  		
  		\node[var] (S1r) at (1r) {$s$};
  		\node[var] (Yr) at (yr) {$x$};
  		\node[var] (Y1r) at (y1r) {$x_1$};
  		\node[var] (Y2r) at (y2r) {$x_2$};
  		\node[var] (SNKr) at (snkr) {$s$};
  		
  		\node at (4.3,0) {$\cdots$};
  		
  		\node[below=4pt] at (S1r.south) {$1$};
  		\node[below=4pt] at (Yr.south) {$2$};
  		\node[below=4pt] at (SNKr.south) {$n{+}k$};
  		
  		\draw[thick] (S1r.north) |- ($(SNKr.north)+(0,1.0)$) -| (SNKr.north);
  		
  		\draw[thick,dotted] (Y1r.north) |- ($(Y1r.north)+(0,0.6)$);
  		\draw[thick, dotted] (Y2r.north) |- ($(Y2r.north)+(0,0.6)$) -| (Yr.north);
  			\draw[decorate, decoration={brace, amplitude=10pt, mirror}]
  		($(y1r.south west)-(0.3,0.3)$) -- ($(dotsr.south east)+ (+0.5,-0.3)$)
  		node[midway, below=9pt] {$\rho$};
  	\end{scope}
  	
  \end{tikzpicture}
  \caption{Graphical representation of the bijection $\Phi$.}
  \label{fig:4}
  \end{figure}
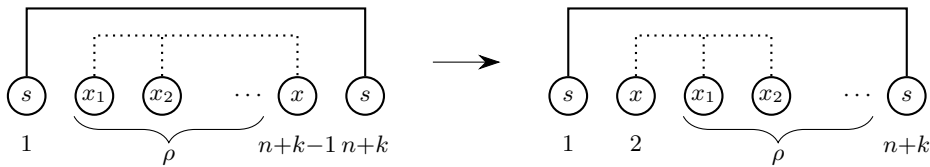
}
  \end{center}
    We define
   $ \Phi(\pi, B, D) = (\pi^\prime, B^\prime, D^\prime)$
   where
   $ B^\prime = ( B \setminus \{n\} ) \cup \{1\}$,
   $ D^\prime = D \setminus \{n\} $
   and
   $ \pi^\prime  = (1, n+k) \cup \gamma \circ \pi_{| \{2, 3, \dots, n+k-1\}} \circ \gamma^{-1}$ 
   for
    $ \gamma = (2, 3, \dots, n+k-1)$ (see the figure above).
    With these notations, we have that
    $ \pi^\prime $ is non-crossing (see Property ...),
    that
     $ | D^\prime | = | D | + 1 $ and, if
    \begin{align*}
    \big( y_1, y_2, \dots, y_{n+k}\big)
    = & \big(
    (s, x)_{|\iota_{B^\prime}(B^\prime \setminus D^\prime)},
    (x, s)_{| \iota_{B^\prime}(D^\prime)}, s_{|\iota_{B^\prime} ([n]\setminus B^\prime)}
    \big)\\
     \big( x_1, x_2, \dots, x_{n+k}\big)
    = & \big(
    (s, x)_{|\iota_{B}(B \setminus D)},
    (x, s)_{| \iota_{B}(D)}, s_{|\iota_{B} ([n]\setminus B)}
    \big)
    \end{align*} 
   then
   $x_1 = x_{n+k} =s= y_1 = y_{n+k} $, 
   while
   $ x_u = y_{\gamma(u)} $ for $ 1 < u < n+k $,
   which give that
    $ E (\pi, B, D) = - E(\pi^\prime, B^\prime, D^\prime)$.
  
  Finally, suppose that
   $ (\pi, B, D) \in \mathcal{A}_1 $
    is such that $(1, n+k)$ is not a block of $\pi$ containing elements from 
    $ \iota_B([n]\setminus B)$
     and let $(i, i+1) $ be the first (in lexicographic order) block of $ \pi$ containing elements of
   $ \iota_B ([n]\setminus B )$. From the definition of $ \mathcal{A}_1 $, we have that 
   $ i \in \iota_B ([n]\setminus B) $, 
   so $ x_i = x_{i+1}= s $, $ x_{i+2} = x $, and
   $ (i,i+1) = \iota_B (\{j\}) $, for some $j \in B \setminus D $.
   \begin{center}
   	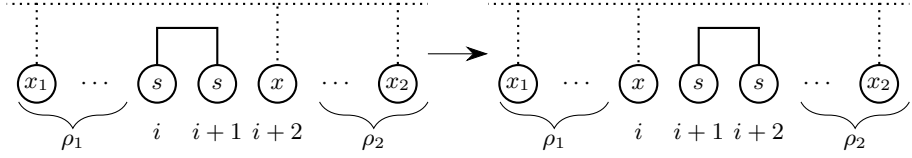
\begin{figure}[H]
  \begin{tikzpicture}[
  	scale=0.8,
  	baseline=(current bounding box.center),
  	every node/.style={font=\small},
  	var/.style={
  		draw,
  		circle,
  		minimum size=5mm,
  		inner sep=1pt,
  		thick
  	},
  	>=Stealth
  	]
  	
  	\begin{scope}[xshift=0cm]
  		\coordinate (y1) at (0,0);
  		\coordinate (dots1) at (1.0,0);
  		\coordinate (s1) at (2.0,0);
  		\coordinate (s2) at (3.0,0);
  		\coordinate (y) at (4.0,0);
  		 \coordinate (dots2) at (5.0,0);
  		\coordinate (y2) at (6.0,0);
  		
  		\node[var] (Y1) at (y1) {$x_1$};
  		\node[var] (S1) at (s1) {$s$};
  		\node[var] (S2) at (s2) {$s$};
  		\node[var] (Y) at (y) {$x$};
  		\node[var] (Y2) at (y2) {$x_2$};
  		
  		\node at (1.0,0) {$\cdots$};
  		\node at (5.0,0) {$\cdots$};
  		\node[below=4pt] at (S1.south) {$i$};
  		\node[below=4pt] at (S2.south) {$i+1$};
  		\node[below=4pt] at (Y.south) {$i+2$};
  		
  		\draw[thick] (S1.north) |- ($(S2.north)+(0,0.6)$) -| (S2.north);
  		
  		\draw[thick,dotted] (Y1.north) |- ($(Y.north)+(0,1)$) -| (Y.north);
  		\draw[thick, dotted] (Y2.north) |- ($(Y.north)+(0,1)$) ;
  		\draw[thick, dotted] 
  		($(Y2.north)+(0,1)$) -- ++(0.5,0);
  		\draw[thick, dotted] 
  		($(Y1.north)+(0,1)$) -- ++(-0.5,0);
  		\draw[decorate, decoration={brace, amplitude=10pt, mirror}]
  		($(y1.south west)-(0.3,0.3)$) -- ($(dots1.south east)+ (+0.5,-0.3)$)
  		node[midway, below=9pt] {$\rho_1$};
  		\draw[decorate, decoration={brace, amplitude=10pt, mirror}]
  		($(dots2.south west)-(0.3,0.3)$) -- ($(y2.south east)+ (+0.5,-0.3)$)
  		node[midway, below=9pt] {$\rho_2$};
  	\end{scope}

  	\draw [-{Stealth[length=3mm, width=2mm]}] (6.5,0.5) -- (7.5,0.5);

  	\begin{scope}[xshift=8cm]
  			\coordinate (y1) at (0,0);
  		\coordinate (dots1) at (1.0,0);
  		\coordinate (y) at (2.0,0);
  		\coordinate (s1) at (3.0,0);
  		\coordinate (s2) at (4.0,0);
  		\coordinate (dots2) at (5.0,0);
  		\coordinate (y2) at (6.0,0);
  		
  		\node[var] (Y1) at (y1) {$x_1$};
  		\node[var] (S1) at (s1) {$s$};
  		\node[var] (S2) at (s2) {$s$};
  		\node[var] (Y) at (y) {$x$};
  		\node[var] (Y2) at (y2) {$x_2$};
  		
  		\node at (1.0,0) {$\cdots$};
  		\node at (5.0,0) {$\cdots$};
  		\node[below=4pt] at (Y.south) {$i$};
  		\node[below=4pt] at (S1.south) {$i+1$};
  		\node[below=4pt] at (S2.south) {$i+2$};
  		
  		\draw[thick] (S1.north) |- ($(S2.north)+(0,0.6)$) -| (S2.north);
  		
  		\draw[thick,dotted] (Y1.north) |- ($(Y.north)+(0,1)$) -| (Y.north);
  		\draw[thick, dotted] (Y2.north) |- ($(Y.north)+(0,1)$) ;
  		\draw[thick, dotted] 
  		($(Y2.north)+(0,1)$) -- ++(0.5,0);
  		\draw[thick, dotted] 
  		($(Y1.north)+(0,1)$) -- ++(-0.5,0);
  	\draw[decorate, decoration={brace, amplitude=10pt, mirror}]
  	($(y1.south west)-(0.3,0.3)$) -- ($(dots1.south east)+ (+0.5,-0.3)$)
  	node[midway, below=9pt] {$\rho_1$};
  	\draw[decorate, decoration={brace, amplitude=10pt, mirror}]
  	($(dots2.south west)-(0.3,0.3)$) -- ($(y2.south east)+ (+0.5,-0.3)$)
  	node[midway, below=9pt] {$\rho_2$};
  	\end{scope}
  \end{tikzpicture}
  \caption{Graphical representation of bijection $\Phi$.}
  \label{fig:5}
  \end{figure}
\end{center}
  In this case we define 
  $ \Phi(\pi, B, D) = (\pi^\prime, B^\prime, D^\prime)$ 
  where
  $ B^\prime = (B \setminus \{j \}) \cup \{ j-1\} $,
  $ D^\prime = D \cup\{j-1\}$, and 
  $ \pi^\prime $ constructed as follows: if $ \pi $ has the blocks $(i, i+1), B_1, B_2, \dots, B_l $  such that 
 $ i+ 2 \in B_1 $, then 
 $\pi^\prime $ has the blocks $ (i+1, i+2), B_1^\prime , B_2, \dots, B_l $ where
 $ B_1^\prime = B_1 \setminus \{ i+2\} \cup \{i\} $ (see Figure \ref{fig:5}).
 
  With these notations, we have that $ \pi^\prime $ is non-crossing (as $ \pi $ without the block $ (i, i+1)$ is the same as $ \pi^\prime$ without the block $ (i+1, i+2)$ ), that $| D^\prime | = | D | + 1 $ and that if 
     \begin{align*}
  \big( y_1, y_2, \dots, y_{n+k}\big)
  = & \big(
  (s, x)_{|\iota_{B^\prime}(B^\prime \setminus D^\prime)},
  (x, s)_{| \iota_{B^\prime}(D^\prime)}, s_{|\iota_{B^\prime} ([n]\setminus B^\prime)}
  \big)\\
  \big( x_1, x_2, \dots, x_{n+k}\big)
  = & \big(
  (s, x)_{|\iota_{B}(B \setminus D)},
  (x, s)_{| \iota_{B}(D)}, s_{|\iota_{B} ([n]\setminus B)}
  \big)
  \end{align*} 
  then $ x_u = y_u $ for 
  $ u \notin\{i, i+1, i+1\} $ 
  and
   $ y_{i}= x = x_{i+2} $,
    while
   $ y_{i+1}=y_{i+2}= s = x_i$, 
   which give that 
  $ E (\pi, B, D) = - E(\pi^\prime, B^\prime, D^\prime)$
  also in this case.
\end{proof}

\begin{cor}
  If $ x $ is freely infinitely divisible, then $ s + i[s, x] $ is also freely infinitely divisible.
\end{cor}


\section{Free cumulants of $x+i[x,s]$}

Recall the notation from the Introduction . For a fixed interval partition on \( n \) elements, \( \sigma \in \mathrm{Int}(n) \), with \( |\sigma| = k \) denote the number of blocks of \( \sigma \).  
For such a partition \( \sigma \), consider a non-crossing partition \( \pi \in \mathrm{NC}(k) \).  
We denote by \( \pi(\sigma) \) the non-crossing partition \( \rho \in \mathrm{NC}(n) \) whose blocks are unions of the blocks of \( \sigma \) corresponding to indices in the same block of \( \pi \).  
More precisely, if \( \sigma = \{W_1, \ldots, W_k\} \) and \( \pi = \{V_1, \ldots, V_m\} \), then
\[
\rho = \pi(\sigma) = \{Z_1, \ldots, Z_m\}, \quad \text{where } Z_i = \bigcup_{j \in V_i} W_j, \; i = 1,\ldots,m.
\]
See Figure \ref{fig:1} in the Introduction for and illustration of how $\pi(\sigma)$ is defined.
The main result of this section is the following Proposition. 

\begin{prop}\label{prop:41}
	Let \( x, s \) be free random variables, and suppose \( s \) is semicircular. Then
	\begin{align}\label{eq:41}
	\kappa_n(x+i[x,s])
	=\kappa_n(x)+ \sum_{k=1}^{\lfloor n/2 \rfloor} 
	\sum_{\substack{\sigma \in \mathrm{Int}_{\geq 2}(n) \\ |\sigma|=k,\, 1 \in V_1}}
	\sum_{\pi \in \mathrm{NC}(k)} 
	|V_1|\, \kappa_{\pi(\sigma)}(x).
	\end{align}
\end{prop}

\begin{proof}
			We will show that, when expanding $\kappa_n(x+i[x,s])$
	(first by linearity, and then by application of products as entries formula from Theorem \ref{thm:product-entries}) only contributions from non--crossing partitions of the form $\pi(\sigma)$ described in \eqref{eq:41} may not cancel.
 For other partitions, the contribution value might be non-zero but it is canceled by the contribution of some other partition. 
	Any partition $\pi$ whose contribution may not cancel is (after a cyclic permutation) such that for \[\kappa_{\pi}(s,\underbrace{x,\ldots,x}_{V_1},s,s,\underbrace{x,\ldots,x}_{V_2},s,s,\ldots,s,s,\underbrace{x,\ldots,x}_{V_k},s),
	\]
	 we have that neighbouring $s$ are paired by $\pi$ and all $x$ between pairs of $s$ are in the same block and blocks of $x$ can be merged in a non--crossing way. Thus we obtain the interval partition $\sigma=\{V_1,\ldots,V_k\}$ and all non--crossing partitions of blocks of $\sigma$ which gives the final result. The factor $|V_1|$ comes from the fact that free cumulants are invariant by cyclic permutation, and we have exactly $|V_1|$ such shifts, with the last one being
	 \[\kappa_{\pi}(x,s,s,\underbrace{x,\ldots,x}_{V_2},s,s,\ldots,s,s,\underbrace{x,\ldots,x}_{V_k},s,s,\underbrace{x,\ldots,x}_{V_1-1}).
	 \]
	 Further shifts are countained in shifts of other interval partitions form $\mathrm{Int}(n)$. For example shifts of second block are contained in the interval partition $\sigma$ where the first block has cardinality $|V_2|$ and last block has cardinality $|V_1|$.

	With the notations from Section \ref{Section:3} we have that

\begin{align*}
\kappa_n\big(x +i[x,s]\big) = & 
 \sum_{B \subseteq [n] }
 i^{|B|}\kappa_n \big( [x, s]_{| B }, x_{| [n]\setminus B }  )\\
 =
 \sum_{B \subseteq [n] }  &
  \sum_{D\subseteq B }
  \sum_{\substack{\pi \in NC(n+k) \\ \pi \vee \tau_B = \mathbbm{1}_{n+k} } } 
  E (\pi, B, D)
\end{align*}	
 where \( E(\pi, B, D) =  i^{|B|}(-1)^{| D|}  \kappa_{\pi} \big[ 
(x, s)_{ | \iota_B (B \setminus D)} , (s, x)_{| \iota_B (D)} , x_{| \iota_B ([n] \setminus B )}\big].
\)

First we shall identify terms that cancel in the above summation. For this, we need to introduce several notations. 
 
Assume that 
$ B \subseteq [ n ]$ 
is such that $ | B | = k $. Denote
 \[(x_1, x_2, \dots, x_{n+ k } ) = 
\big(
(x, s)_{ | \iota_B (B \setminus D)} , (s, x)_{| \iota_B (D)} , x_{| \iota_B ([n] \setminus B )}
\big).
 \] 
 
 To simplify the terminology, we shall call $s$-\textit{blocks} of $(\pi, B, D)$ the blocks of $ \pi $ that contain some $j $ with $ x_j =s $.
 
For
 $ E(\pi, B, D) \neq 0$ 
 we need all $s$-blocks  to be pairs $ (j, l)$ with $x_j = x_l = s $ (in  particular $ | B |$ must be even). 
 Moreover, if $ j < l $ and
  $ j \in \iota_{B} (B \setminus D) $ 
  and $ l \in \iota_{B} (D) $, then $ l= j+1$, 
  otherwise 
  $\{ j+1, j+2, \dots, l-1\} \trianglelefteq \pi\vee \tau_B $; also, if 
  $ j \in \iota_{B}(D) $ 
  and 
  $ l \in \iota_{B}(B \setminus D)$
  then $ j=1$ and $ l=n+k $, otherwise
  $ \{ j, j+1, \dots, l\} \trianglelefteq \pi\vee \tau_B $.
  
   Next, if $ (j, l)$ is a $s$-block of $\pi $, we define the non-crossing partition $\pi_{(j)}$ by replacing in $\pi $ the block $(j, l)$ with singletons $ (j)$ and $(l)$. Denote by \textrm{Bl}$(j)$, respectively by \textrm{Bl}$(l)$, the blocks of
    $ \tau_B \vee \pi_{(j)} $
    containing $ j $, respectively $ l $. 
    Since $(j, l)$ is not a block of $ \tau_B $, we have that
    $ \textrm{Bl}(j) \cup \textrm{Bl}(l) = [n+k]$, which furthermore gives that (after a cyclic permutation) \textrm{Bl}$(j)$ and \textrm{Bl}$(l)$ are interval blocks.
    In order to show that these are the only possibilities for $s$--blocks we need to exclude the possibilities of $j,l$ both being in $\iota_{B}(D)$ or in $\iota_{B}(B\setminus D)$.
    
    With \textrm{Bl}$(j)$ and \textrm{Bl}$(l)$  We shall organize the $s$-blocks $j, l$ with $ j < l $ into 5 disjoint classes:
    \begin{enumerate}
    	\item[$\bullet$] type 1:
    	\textrm{Bl}$(j)$ = \textrm{Bl}$(l) = [n+k]$ 
    	and  $l = j+1 $
    	or $ j=1$, $ l = n+k $;
    	\item[$\bullet$] type 2:
    	$ l = j+1$,  \textrm{Bl}$(j) \neq $\textrm{Bl}$(l)$
    	and \textrm{Bl}$(j) \subseteq \{1, 2, \dots, j\}$;   
    \item[$\bullet$] type 3: 
    $ l = j+1$,  
    \textrm{Bl}$(j) \neq $\textrm{Bl}$(l)$ 
    and \textrm{Bl}$(l) \subseteq \{ l, l+1, \dots, n-1\}$
    or \\
    ${}$\hspace{1cm} $ j =1 $, $ l = n+k $ and \textrm{Bl}$(j) \neq $ \textrm{Bl}$(l)$;
     \item[$\bullet$] type 4:
      $l > j+1 $ and 
      \textrm{Bl}$(j) = \{j, j+1, \dots, l-1\}$;
      \item[$\bullet$] type 5:
      $l > j+1 $ and 
      \textrm{Bl}$(l) = \{ j+1, j+2, \dots, l \} $.
 \end{enumerate}
 
    On the $s$-blocks of $(\pi, B, D)$ we shall consider 2 order relations: 
    \begin{enumerate}
    	\item[(i)] the \emph{left-to right lexicographic order}: $(a, b) $ is less that $(c, d)$ if $ a < c $;
    	\item[(ii)] the \emph{right-to-left lexicographic order}: $ (a, b) $ is less that $(c, d)$ if $ b < d $.
    \end{enumerate}
    
    Finally, we shall organize the triples $(\pi, B, D)$ as above into 5 disjoint sets. We say that $(\pi, B, D)$ is
    \begin{enumerate}
    	\item[$\bullet$] in the set $\mathcal{A}_1$ if all $s$-blocks of $(\pi, B, D)$ are of type 1;
    	\item[$\bullet$] in the set $\mathcal{A}_2$ if 
    	$(\pi, B, D)$ contains some $s$-blocks of type 2 or of type 4, but the first, in \emph{right-to-left lexicographic order} of these blocks is of type 2
    	\item[$\bullet$] in the set $\mathcal{A}_3$ if
    	$(\pi, B, D)$ does not contain $s$-blocks of type 2 or of type 4, but contains an $s$-blocks of type 3 or of type 5 and the first of these $s$-blocks, in \emph{left-to-right lexicographic order} is of type 3.
    	\item[$\bullet$] in the set $\mathcal{A}_4 $ if 
    	$(\pi, B, D)$ contains some $s$-blocks of type 2 or of type 4, but the first, in \emph{right-to-left lexicographic order} of these blocks is of type 4
    	\item[$\bullet$] in the set $\mathcal{A}_5$ if
    	$(\pi, B, D)$ does not contain $s$-blocks of type2 not of type 4, but contains $s$-blocks of type 3 or of type 5 and the first of these $s$-blocks, in \emph{left-to-right lexicographic order} is of type 5.
    \end{enumerate}

  For $ v \in\{2, 3\}$ we shall define the bijections 
  $ \Phi_v:\mathcal{A}_v \rightarrow \mathcal{A}_{v+2}$
  such that
  $ E(\pi, B, D) = - E(\Phi_v(\pi, B, D))$.
  
  We define $\Phi_2$ as follows. Let $(\pi, B, D) \in \mathcal{A}_2 $ and let $ (j, j+1)$ be its first, in the \emph{right-to-left lexicographic order}, $s$-block
  which of type 2 or of type 4. Then $(j, j+1)$ is of type 2 and let \textrm{Bl}$(j)= \{ u, u+1, \dots, j\}$. From the definition of \textrm{Bl}$(j)$, we have that $\{ u, u+1, \dots,j, j+1\} \trianglelefteq \pi$.
  
  Note that $ x_j = x_{j+1} = s $ (since $(j, j+1) $ is an $s$-block), hence $ x_{j-1} = x_{j+2} = x $.  Let $ a = \iota_b^{-1}(u) $, $b = \iota_b^{-1}(j)$. We have that $ b \in B \setminus D $ since $\iota_B(b) =\{ j-1, j\} $ with $ x_{j-1} = x $ and $ x_j = S $.
  
  Define 
  \begin{align*}
  B^\prime &= B \cap \big( [n]\setminus\{ a, a+1, \dots, b\}\big) \cup \{a\} \cup \lambda \big( B \cap \{ a, a+1, \dots, b-1\}\big) \\
  D^\prime &=
  D \cap \big( [n]\setminus\{ a, a+1, \dots, b\}\big)
  \cup \{ a\} \cup \lambda \big( D \cap \{ a, a+1, \dots, b-1\}\big)
  \end{align*}
  where $\lambda$ is the right-translation $ \lambda(t) = t+1$.
  
   Since $ b = \iota_B^{-1} (j) \in B \setminus D $, we have that
  \begin{align*}
  |B^\prime| & = | B \cap \big([n]\setminus \{b\} \big) | + 1 = | B |\\
  |D^\prime | & = | D \cap \big([n]\setminus \{b\} \big) | + 1 = | D | + 1,
  \end{align*}
  hence 
  \begin{equation}\label{eq:2-4}
   i^{| B^\prime |}(-1)^{| D^\prime|} = - i^{| B |}(-1)^{| D|}.
   \end{equation}
  
  Furthermore define $\pi^\prime $ as follows.  The restrictions of $\pi $ and $\pi^\prime $ to the set  $ [n+k]\setminus \{ u, u+1, \dots, j+1\}$ coincide. If $\pi_{| \{ u, u+1, \dots, j+1\}} $ has blocks $ B_0, B_1, B_2, \dots, B_q $ such that $ B_0 = (j, j+1)$ and $ B_1 $ contains $ j-1$, then $\pi^\prime_{| \{ u, u+1, \dots, j+1\}} $ has blocks $ (u, j+1)$, $\lambda^2(B_1 \setminus \{ j-1\}) \cup \{u+1 \}$, $ \lambda^2(B_2)$, \dots, $\lambda^2(B_q)$, see Figure \ref{Fig:5} for graphical representation of $\Phi_2$.
  \begin{center}
  	{
  		\setlength{\intextsep}{-4pt}
  	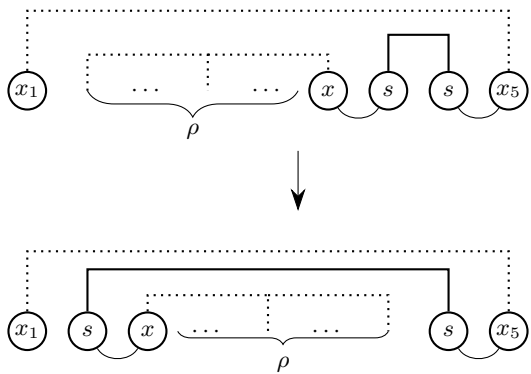
\begin{figure}[H]
  	\begin{tikzpicture}[
  		scale=0.8,
  		baseline=(current bounding box.center),
  		every node/.style={font=\small},
  		var/.style={
  			draw,
  			circle,
  			minimum size=5mm,
  			inner sep=1pt,
  			thick
  		},
  		>=Stealth
  		]
  		
  		\begin{scope}[yshift=0cm]
   			\coordinate (x1) at (0,0);
			 \coordinate (x2) at (1.0,0);
			 \coordinate (dots1) at (2.0,0);
			 \coordinate (x3) at (3.0,0);
			 \coordinate (dots2) at (4.0,0);
			 \coordinate (x4) at (5.0,0);
			 \coordinate (s1) at (6.0,0);
			 \coordinate (s2) at (7.0,0);
			 \coordinate (x5) at (8.0,0);
			 \node[var] (X1) at (x1) {$x_1$};
			 \node[var] (X4) at (x4) {$x$};
			 \node[var] (S1) at (s1) {$s$};
			 \node[var] (S2) at (s2) {$s$};
 				\node[var] (X5) at (x5) {$x_5$};
			 
			 \node at (2.0,0) {$\cdots$};
			 \node at (4.0,0) {$\cdots$};
  			
  			
  			\draw (X4) to[out=300, in=240] (S1);
  			\draw (S2) to[out=300, in=240] (X5);
  			\draw[thick] (S1.north) |- ($(S2.north)+(0,0.6)$) -|(S2.north);
  			\draw[thick, dotted] (X1.north)	|- ($(X5.north)+(0,1)$) -|(X5.north);
  			\draw[thick, dotted] (x2.north)|-
  			($(x3.north)+(0,0.6)$) -|(x3.north);
  			\draw[thick, dotted] 
  			($(x3.north)+(0,0.6)$) -|(X4.north);
  			\draw[decorate, decoration={brace, amplitude=10pt, mirror}]
  			(x2.south west) -- ($(x4.south east)+ (-0.5,0)$)
  			node[midway, below=9pt] {$\rho$};
  		\end{scope}

  		\begin{scope}[yshift=-2cm]
\draw [-{Stealth[length=3mm, width=2mm]}] (4.5,1) -- (4.5,0);
  	\end{scope}

  		\begin{scope}[yshift=-4cm]
	\coordinate (x1) at (0,0);
	\coordinate (s1) at (1.0,0);
	\coordinate (x2) at (2.0,0);
	\coordinate (dots) at (3.0,0);	
	\coordinate (x3) at (4.0,0);
	\coordinate (dots) at (5.0,0);
	\coordinate (x4) at (6.0,0);
	\coordinate (s2) at (7.0,0);
	\coordinate (x5) at (8.0,0);
	\node[var] (X1) at (x1) {$x_1$};
	\node[var] (X2) at (x2) {$x$};
	\node[var] (S1) at (s1) {$s$};
	\node[var] (S2) at (s2) {$s$};
	\node[var] (X5) at (x5) {$x_5$};
	
	\node at (3.0,0) {$\cdots$};
	\node at (5.0,0) {$\cdots$};
	
	
	\draw (S1) to[out=300, in=240] (X2);
	\draw (S2) to[out=300, in=240] (X5);
	\draw[thick] (S1.north) |- ($(S2.north)+(0,0.7)$) -|(S2.north);
	\draw[thick, dotted] (X1.north)	|- ($(X5.north)+(0,1)$) -|(X5.north);
	\draw[thick, dotted] (X2.north)|-
	($(x3.north)+(0,0.6)$) -|(x3.north);
	\draw[thick, dotted] 
	($(x3.north)+(0,0.6)$) -|(x4.north);
	\draw[decorate, decoration={brace, amplitude=5pt, mirror}]
	($(x2.south west)+(0.5,0)$) -- (x4.south east)
	node[midway, below=6pt] {$\rho$};
\end{scope}
  		
%
  			
  			
  		
  	\end{tikzpicture}
  	\caption{Graphical representation of bijection $\Phi_2$.}
  	\label{Fig:5}
  	\end{figure}
  }
  \end{center}
  
   The block containing $u $ in $ \pi^{\prime}_ {(u)} \vee \tau_{B^\prime}$ is then $\{u, u+1, \dots, j\} $, hence $ (u, j+1)$ is of type 4 in 
   $(\pi^\prime, B^\prime, D^\prime)$.
    Moreover, all $s$-blocks of $(\pi^\prime, B^\prime, D^\prime)$ less than $ (u, j+1)$ are $s$-blocks of $(\pi, B, D)$ less than $(j, j+1)$, hence not of type 2 or 4, so $(\pi^\prime, B^\prime, D^\prime)\in \mathcal{A}_4$.

  Finally, if $(y_1, \dots, y_{n+k}) =
  \big(
  (x, s)_{ | \iota_{B^\prime} (B^\prime \setminus D^\prime)} , (s, x)_{| \iota_{B^\prime} (D^\prime)} , x_{| \iota_{B^\prime}([n] \setminus {B^\prime} )}
  \big) $
  then we have that 
  $y_t = x_t$ for $ t \notin \{u, u+1, \dots, j\}$,
  $y_u = x_j $, $y_{u+1} = x_{j-1} $ and $ y_{v+2} = x_v $ for $ v \in \{ u, \dots, j-2\}$, 
  therefore
  $\kappa_\pi\big[x_1, x_2, \dots, x_{n+k}] = \kappa_{\pi^\prime}[ y_1, y_2, \dots, y_{n+k}]$, 
  which, together with (\ref{eq:2-4}), gives that
  $ E(\pi, B, D) = -E(\pi^\prime, B^\prime, D^\prime)$.

  The construction of $\Phi_3 $ is similar. Let 
  $(\pi, B, D) \in \mathcal{A}_3 $. First we consider the case when 
   its first, in the
  \emph{left-to-right lexicographic order}, $s$-block of type 3 if of the form 
  $(j, j+1)$ with $j+1 < n+k$.
  Let
   \textrm{Bl}$(j+1) = \{j+1, j+2, \dots, u-1, u\}$. 
   Again, from the definition of \textrm{Bl}$(j+1)$,  we have that 
  $\{j, j+1, \dots, u\}  \trianglelefteq \pi$.
  
  Note that 
  $ x_j = x_{j+1}= s $
   and 
   $ x_{j-1}=x_{j+2} = x $. 
   Therefore, if 
   $ a = \iota_b^{-1}(u)$
    and  
  $ b = \iota_B^{-1}(j+1)$, 
  we have that 
  $b \in  D$. 
 
  Define 
  \begin{align*}
  B^\prime &= B \cap \big( [n]\setminus \{ b, b+1, \dots, a\}\big) \cup \{ a\} \cup \lambda\big(B \cap\{ b, b+1, \dots, a-1\}\big) \\
  D^\prime & = D \cap \big( [n] \setminus \{ b, b+1, \dots, a\} \big) \cup \lambda\big(D \cap\{ b, b+1, \dots, a-1\}\big) 
  \end{align*}
  since $ b \in D $, we have that $ |B^\prime| = | B | $ 
  and $ | D^\prime | = |D | -1 $, hence
  \begin{equation}\label{eq:3-5}
  i^{| B^\prime |}(-1)^{| D^\prime|} = - i^{| B |}(-1)^{| D|}
  \end{equation}

  In this case, we define $\pi^\prime$ as follows. We let
  $\pi^\prime_{| [n+k] \setminus \{j, j+1, \dots, u\}} = 
 \pi_{| [n+k] \setminus \{j, j+1, \dots, u\}} $
 and if
 $ \pi_{| \{j, j+1, \dots, u\} } $ 
 has blocks
 $ B_0, B_1, \dots, B_q $ 
 such that 
 $ B_0 = (j, j+1)$ 
 and $ B_1$ contains $j+2$, then
 $\pi^\prime_{| \{ j, j+1, \dots, u\}}$
 has blocks
 $(j, u)$, $\lambda^2 \big(B_1 \setminus\{ j+2\} ) \cup \{ u-1 \} $, 
 $\lambda^2(B_2)$, \dots, $\lambda^2(B_q)$, see the Figure \ref{fig:6} for graphical representation of $\Phi_3$.
   \begin{center}
   	{
   	\setlength{\intextsep}{-3pt}
   	\begin{figure}[H]
   	\begin{tikzpicture}[
   		scale=0.8,
   		baseline=(current bounding box.center),
   		every node/.style={font=\small},
   		var/.style={
   			draw,
   			circle,
   			minimum size=5mm,
   			inner sep=1pt,
   			thick
   		},
   		>=Stealth
   		]
   		
   		\begin{scope}[yshift=0cm]
   			\coordinate (x1) at (0,0);
   			\coordinate (s1) at (1.0,0);
   			\coordinate (s2) at (2.0,0);
   			\coordinate (x2) at (3.0,0);
   			\coordinate (dots1) at (4.0,0);
   			\coordinate (x3) at (5.0,0);
   			\coordinate (dots2) at (6.0,0);
   			\coordinate (x4) at (7.0,0);
   			\coordinate (x5) at (8.0,0);
   			\node[var] (X1) at (x1) {$x_1$};
   			\node[var] (X2) at (x2) {$x$};
   			\node[var] (S1) at (s1) {$s$};
   			\node[var] (S2) at (s2) {$s$};
   			\node[var] (X5) at (x5) {$x_5$};
   			
   			\node at (4.0,0) {$\cdots$};
   			\node at (6.0,0) {$\cdots$};
   			
   			
   			\draw (X1) to[out=300, in=240] (S1);
   			\draw (S2) to[out=300, in=240] (X2);
   			\draw[thick] (S1.north) |- ($(S2.north)+(0,0.6)$) -|(S2.north);
   			\draw[thick, dotted] (X1.north)	|- ($(X5.north)+(0,1)$) -|(X5.north);
   			\draw[thick, dotted] (X2.north)|-
   			($(x3.north)+(0,0.6)$) -|(x3.north);
   			\draw[thick, dotted] 
   			($(x3.north)+(0,0.6)$) -|(x4.north);
   			\draw[decorate, decoration={brace, amplitude=10pt, mirror}]
   			($(x2.south east)+ (0.5,0)$) -- (x4.south east)
   			node[midway, below=9pt] {$\rho$};
   		\end{scope}

   		\begin{scope}[yshift=-2cm]
   			\draw [-{Stealth[length=3mm, width=2mm]}] (4.5,1) -- (4.5,0);
   		\end{scope}
   		
   		\begin{scope}[yshift=-4cm]
   			\coordinate (x1) at (0,0);
   			\coordinate (s1) at (1,0);
   			\coordinate (s2) at (7.0,0);
   			\coordinate (x2) at (6.0,0);
   			\coordinate (dots1) at (3.0,0);
   			\coordinate (x3) at (2.0,0);
   			\coordinate (dots2) at (3.0,0);
   			\coordinate (x4) at (4.0,0);
   			\coordinate (x5) at (8.0,0);
   			\node[var] (X1) at (x1) {$x_1$};
   			\node[var] (X2) at (x2) {$x$};
   			\node[var] (S1) at (s1) {$s$};
   			\node[var] (S2) at (s2) {$s$};
   			\node[var] (X5) at (x5) {$x_5$};
   			
   			\node at (3.0,0) {$\cdots$};
   			\node at (5.0,0) {$\cdots$};
   			
   			
   			\draw (X1) to[out=300, in=240] (S1);
   			\draw (S2) to[out=240, in=300] (X2);
   			\draw[thick] (S1.north) |- ($(S2.north)+(0,0.6)$) -|(S2.north);
   			\draw[thick, dotted] (X1.north)	|- ($(X5.north)+(0,1)$) -|(X5.north);
   			\draw[thick, dotted] (X2.north)|-
   			($(x3.north)+(0,0.6)$) -|(x3.north);
   			\draw[thick, dotted] 
   			($(x4.north)+(0,0.6)$) -|(x4.north);
   			\draw[decorate, decoration={brace, amplitude=10pt, mirror}]
   			($(x3.south east)$) -- ($(x2.south east)+ (-0.5,0)$)
   			node[midway, below=9pt] {$\rho$};
   		\end{scope}
   	\end{tikzpicture}
   	\caption{Graphical representation of bijection $\Phi_3$.}
   	\label{fig:6}
   	\end{figure}
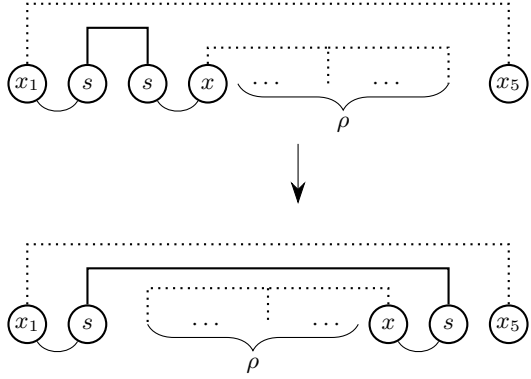
   }
   \end{center}
 The block containing $u $ in $ \pi^{\prime}_ {(u)} \vee \tau_{B^\prime}$ is then $\{j, j+1, \dots, u\} $, hence $ (j, u)$ is of type 5 in 
 $(\pi^\prime, B^\prime, D^\prime)$.
 All $s$-blocks of $(\pi^\prime, B^\prime, D^\prime)$ less than $(j, u)$ are $s$-blocks of $(\pi, B, D)$ less than $(j, j+1)$, hence of type 1, so $(\pi^\prime, B^\prime, D^\prime) \in \mathcal{A}_5$.
 
 If $(y_1, \dots, y_{n+k}) =
 \big(
 (x, s)_{ | \iota_{B^\prime} (B^\prime \setminus D^\prime)} , (s, x)_{| \iota_{B^\prime} (D^\prime)} , x_{| \iota_{B^\prime}([n] \setminus {B^\prime} )}
 \big) $
 then we have that 
 $y_t = x_t$
 for $ t \notin \{j+1, j+2, \dots, u\} $,
 $ y_u = x_{j+1}$,
 $y_{u_1}= x_{j+2} $
 and 
 $ y_{v-2}= x_v $
 for 
 $ v \in \{j+3, \dots, u-1, u\}$,
 therefore
 $\kappa_\pi\big[x_1, x_2, \dots, x_{n+k}] = \kappa_{\pi^\prime}[ y_1, y_2, \dots, y_{n+k}]$, which, together with (\ref{eq:3-5}), gives that
 $ E(\pi, B, D) = -E(\pi^\prime, B^\prime, D^\prime)$.
 
 For the case $(\pi, B, D) \in \mathcal{A}_3 $ and $ (1, n+k)$  $s$-block of type 3, let 
 \textrm{Bl}$(1) = \{1, 2, \dots, u\}$.
  Noting that $u < n +k-1 $, that $1 \in  D $
  and letting $a = \iota_B(u)$, we define
   $(\pi^\prime, B^\prime, D^\prime)$ 
   as follows.
   \begin{align*}
   B^\prime & = 
   B \cap \big([n]\setminus \{ 1, 2,\dots, a \}  \big) 
   \cup\{a\} \cup \lambda\big(B \cap \{ 2, 3, \dots, a \}\\
   D^\prime & =
   D \cap \big( [n]\setminus \{ 1, 2,\dots, a \}  \big) 
   \cup \lambda \big( D \cap \{ 2, 3, \dots, a\} \big)
   \end{align*}
  so $ | B^\prime | = | B | $ and 
  $ | D^\prime | = | D | - 1 $.
 
   For $\pi^\prime $, we let
    $\pi^\prime_{| \{u+1, u+2, \dots, n+k-1 \}} = \pi_{| \{u+1, u+2, \dots, n+k-1 \}}$ 
    and if the restriction of $\pi $ to 
    $\{1, 2, \dots, u\} \cup\{n+k\} $
    has the blocks $(1, n+k)$, $B_1, B_2, \dots, B_q$ 
    with
    $2\in B_2$, 
    then 
    $\pi^\prime_{|\{1, 2, \dots, u\} \cup\{n+k\} } $ 
    has the blocks
    $(u, n+k)$, $\lambda^2\big( B_1 \setminus\{2\} \cup\{u-1\} \big)$, $\lambda^2\big(B_2\big), \dots, \lambda^2\big(B_q\big) $ (see Figure \ref{fig:7}) and the rest of the proof is the same as above.
    
       \begin{center}
       	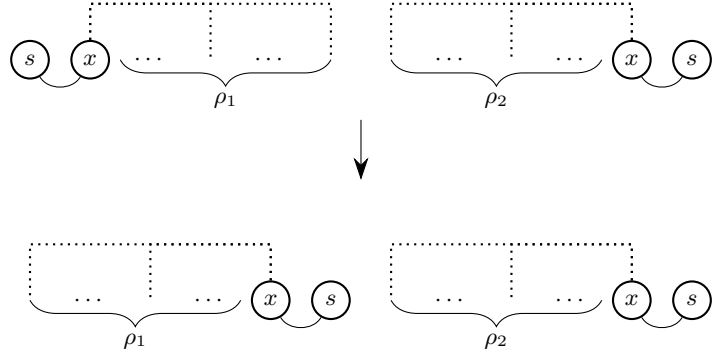
\begin{figure}[H]
    	\begin{tikzpicture}[
    		scale=0.8,
    		baseline=(current bounding box.center),
    		every node/.style={font=\small},
    		var/.style={
    			draw,
    			circle,
    			minimum size=5mm,
    			inner sep=1pt,
    			thick
    		},
    		>=Stealth
    		]
    		

    		\begin{scope}[yshift=0cm]
    			\coordinate (s1) at (0,0);
    			\coordinate (x1) at (1,0);
    			\coordinate (dots1) at (2.0,0);
    			\coordinate (x2) at (3.0,0);
    			\coordinate (dots2) at (4.0,0);
    			\coordinate (x3) at (5.0,0);
    			
    			\coordinate (x4) at (6.0,0);
    			\coordinate (dots3) at (7.0,0);
    			\coordinate (x6) at (8.0,0);
    			\coordinate (dots4) at (9.0,0);
    			\coordinate (x7) at (10.0,0);
    			\coordinate (s2) at (11.0,0);
    			\node[var] (S1) at (s1) {$s$};
    			\node[var] (X1) at (x1) {$x$};
    			\node[var] (X7) at (x7) {$x$};
    			\node[var] (S2) at (s2) {$s$};
    			
    			\node at (dots1) {$\cdots$};
    			\node at (dots2) {$\cdots$};
    			\node at (dots3) {$\cdots$};
    			\node at (dots4) {$\cdots$};
    			
    			
    			\draw (X1) to[out=240, in=300] (S1);
    			\draw (S2) to[out=240, in=300] (X7);
    			\draw[thick,dotted] (X1.north) |- ($(X1.north)+(0,0.6)$)-|(x3.north);
    			\draw[thick,dotted] (X1.north) |- ($(X1.north)+(0,0.6)$)-|(x2.north);
    			\draw[thick,dotted] (X7.north) |- ($(X7.north)+(0,0.6)$)-|(x6.north);
    			\draw[thick,dotted] (X7.north) |- ($(X7.north)+(0,0.6)$)-|(x4.north);
    			\draw[decorate, decoration={brace, amplitude=10pt, mirror}]
    			($(x1.south east)+ (0.5,0)$) -- ($(x3.south east)$)
    			node[midway, below=9pt] {$\rho_1$};
    			\draw[decorate, decoration={brace, amplitude=10pt, mirror}]
    			($(x4.south east)$) -- ($(x7.south east)+ (-0.5,0)$)
    			node[midway, below=9pt] {$\rho_2$};
    		\end{scope}

    		\begin{scope}[yshift=-2cm]
    			\draw [-{Stealth[length=3mm, width=2mm]}] (5.5,1) -- (5.5,0);
    		\end{scope}
    		
    		    		\begin{scope}[yshift=-4cm]
    			\coordinate (s1) at (5,0);
    			\coordinate (x1) at (4,0);
    			\coordinate (dots1) at (1.0,0);
    			\coordinate (x2) at (0.0,0);
    			\coordinate (dots2) at (3.0,0);
    			\coordinate (x3) at (2.0,0);
    			
    			\coordinate (x4) at (6.0,0);
    			\coordinate (dots3) at (7.0,0);
    			\coordinate (x6) at (8.0,0);
    			\coordinate (dots4) at (9.0,0);
    			\coordinate (x7) at (10.0,0);
    			\coordinate (s2) at (11.0,0);
    			\node[var] (S1) at (s1) {$s$};
    			\node[var] (X1) at (x1) {$x$};
    			\node[var] (X7) at (x7) {$x$};
    			\node[var] (S2) at (s2) {$s$};
    			
    			\node at (dots1) {$\cdots$};
    			\node at (dots2) {$\cdots$};
    			\node at (dots3) {$\cdots$};
    			\node at (dots4) {$\cdots$};
    			
    			
    			\draw (X1) to[out=300, in=240] (S1);
    			\draw (S2) to[out=240, in=300] (X7);
    			\draw[thick,dotted] (X1.north) |- ($(X1.north)+(0,0.6)$)-|(x3.north);
    			\draw[thick,dotted] (X1.north) |- ($(X1.north)+(0,0.6)$)-|(x2.north);
    			\draw[thick,dotted] (X7.north) |- ($(X7.north)+(0,0.6)$)-|(x6.north);
    			\draw[thick,dotted] (X7.north) |- ($(X7.north)+(0,0.6)$)-|(x4.north);
    			\draw[decorate, decoration={brace, amplitude=10pt, mirror}]
    			($(x2.south east)$) -- ($(x1.south east)+ (-0.5,0)$)
    			node[midway, below=9pt] {$\rho_1$};
    			\draw[decorate, decoration={brace, amplitude=10pt, mirror}]
    			($(x4.south east)$) -- ($(x7.south east)+ (-0.5,0)$)
    			node[midway, below=9pt] {$\rho_2$};
    		\end{scope}    		
    		
    	\end{tikzpicture}
    	\caption{Graphical representation of $\Phi_3$.}
    	\label{fig:7}
    \end{figure}
    \end{center}

    We will show next that for fixed $B,D$ there is a minimal partition $\pi$ such that $(\pi,B,D)\in \cA_1$. Consider first the case $B=[n]$. After circular permutation we may assume that the cumulant is of the form $\kappa_{n}(xs,sx,xs,sx,\ldots,xs,sx)$. In order for $\pi$ to be in the set $\cA_1$ the partition $\pi$ joins neighboring $s$. Suppose that in $\pi$ position of one of the $x$'s is a singleton, suppose that $x_j=x$ and $j$ is a singleton of $\pi$ then either $x_{j+1}=s$ or $x_{j-1}=s$ depending if $x$ comes form the pair $sx$ or $xs$. Let us consider the case $x_{j+1}=s$ (the case $x_{j-1}=s$ can be treated similarly) in such situation $B(j+1)=\{j,j+1\}$ which is a type 3 block, so we are not in the set $\cA_1$.
    
    Recall that $x_1=x$ and suppose that $x_1$ is not connected with $x_{2n}=x$, then since it is not a singleton, it is connected to some $x_k=x$. Observe that this $x$ cannot come from $sx$, as this would violate the maximality condition $\pi\vee\tau_B=1_{[2n]}$, thus such $x_1$ is connected to $x_j=x$ and $x_{j+1}=s$, since $\pi\in \cA_1$, then $j+1$ and $j+2$ are in the same block of $\pi$, but this leads to the conclusion that $Bl(j+2)$ is a block of type 3, which is contradicts the fact that $\pi\in \cA_1$. We conclude that first $x$ is connected with the last $x$ and by circular permutation of arguments of the free cumulant we conclude that all other $x$'s are joined by $\pi$ with the neighboring $x$. Such a partition clearly belongs to $\cA_1$, as we argued any smaller partition is not in $\cA_1$, moreover any bigger partition is also in $\cA_1$.
    
    Let us now consider the general case when we allow $B\neq [n]$, then we have to consider cumulants, where we allow some number of $x$'s between $sx$ and $xs$, recall that cumulants of the type $\kappa_n(\ldots,xs,x,x,sx,\ldots)$ do not contribute to $E(\pi,B,D)$ for $\pi\in \cA_1$. Suppose that $x_j=s, x_{j+1}=x,x_{j+2}=x,\ldots,x_{j+k}=x,x_{j+k+1}=x,x_{j+k+2}=s$. Then similar argument as before shows that $\pi$ contains a block joining $j+1$ and $j+k+1$, and then since  $\pi\vee \tau_B$ consists of one block we see that in $\pi$ all positions $\{j+1,\ldots,j+k+1\}$ are in the same block, and again we have a minimal partition.    
	
%
%
   \end{proof}

\section{Operator model}

Recall that compactly supported measure $\mu$ is compound free Poisson if and only if there exists compactly supported measure $\rho$ such that the sequence of free cumulants $(\kappa_n(\mu))_{n\geq 1}$ is a sequence of moments of a compactly supported measure. So there exists a compactly supported measure $\rho$ on $\mathbb{R}$ such that
\[\kappa_{n}(\mu)=\int_{\mathbb{R}} t^n d\rho(t),\qquad n\geq 0.\] 
Since dilation of a measure does not affect its infinite divisibility, we will assume that $\kappa_2(\mu)=1$ and thus we will assume that $\rho$ is a compactly supported probability measure. Consider a random variable $Y$ with the distribution $\rho$. We fix a classical probability space $(\Omega,\mathcal{G},\mathbb{P})$ and the space of square integrable random variables $\mathcal{H}=L^2(\Omega,\mathcal{G},\mathbb{P})$, such that $Y\in\mathcal{H}$ we consider the space $\mathcal{F}(\mathcal{H})=\bigoplus_{n=1}^{\infty}  \mathcal{H}^{\otimes n}$ with the inner product 
\[\langle X_1\otimes\ldots\otimes X_n,Y_1\otimes\ldots\otimes Y_m\rangle=\delta_{n,m}\bE(X_1 Y_1)\ldots \bE(X_n Y_n).
\]

 We define operators $\widehat{x}, \widehat{xs},\widehat{sx}$ on $\mathcal{F}(\mathcal{H})$ via
\begin{enumerate}

	\item \begin{align*}
		\widehat{x}\left(Y_1\otimes\ldots\otimes Y_n \right)=\begin{cases}
			Y_1\otimes\ldots\otimes Y_n Y\qquad &\mbox{if $n$ is odd,}\\
			0 \qquad &\mbox{if $n$ is even.}
		\end{cases}
	\end{align*}
	\item \begin{align*}
		\widehat{xs}\left(Y_1\otimes\ldots\otimes Y_n\right)=\begin{cases}
			Y_1\otimes\ldots\otimes Y_n\otimes Y+Y_1\otimes\ldots\otimes Y_{n-1}Y\mathbb{E}(Y_n)\qquad &\mbox{if $n$ is even,}\\
			0 \qquad &\mbox{if $n$ is odd.}
		\end{cases}
	\end{align*}
	\item \begin{align*}
		\widehat{sx}\left(Y_1\otimes\ldots\otimes Y_n\right)=\begin{cases}
			Y_1\otimes\ldots\otimes Y_n Y\otimes 1\qquad &\mbox{if $n=1$,}\\
			Y_1\otimes\ldots\otimes Y_n Y\otimes 1+Y_1\otimes\ldots\otimes Y_{n-1}\mathbb{E}(YY_n)\qquad &\mbox{if $n>1$ and is odd,}\\
			0 \qquad &\mbox{if $n$ is even.}
		\end{cases}
	\end{align*}
\end{enumerate}

  Remark that $ \widehat{x}$ is selfadjoint and that $ \widehat{xs}^\ast = \widehat{sx} $, henceforth 
  $ \widehat{x} + \widehat{sx} + \widehat{sx} $ is selfadjoint.
  
Next, we define the additional operators:
\begin{enumerate}
	\item \begin{align*}
	\widetilde{x}\left(Y_1\otimes\ldots\otimes Y_n \right)=\begin{cases}
		Y_1\otimes\ldots\otimes Y_n Y\qquad &\mbox{if $n$ is even,}\\
		0 \qquad &\mbox{if $n$ is odd.}
	\end{cases}
\end{align*}
\item \begin{align*}
	\widetilde{xs}\left(Y_1\otimes\ldots\otimes Y_n\right)=\begin{cases}
		Y_1\otimes\ldots\otimes Y_n\otimes Y\qquad&\mbox{if $n=1$,}\\
		Y_1\otimes\ldots\otimes Y_n\otimes Y+Y_1\otimes\ldots\otimes Y_{n-1}Y\mathbb{E}(Y_n)\qquad &\mbox{if $n>1$ and odd,}\\
		0 \qquad &\mbox{if $n$ is even.}
	\end{cases}
\end{align*}
\item \begin{align*}
	\widetilde{sx}\left(Y_1\otimes\ldots\otimes Y_n\right)=\begin{cases}
		Y_1\otimes\ldots\otimes Y_n Y\otimes 1+Y_1\otimes\ldots\otimes Y_{n-1}\mathbb{E}(YY_n)\qquad &\mbox{if $n$ is even,}\\
		0 \qquad &\mbox{if $n$ is odd.}
	\end{cases}
\end{align*}
	\end{enumerate}

 We also have that $ \widetilde{x} $ is selfadjoint 
 and that 
 $ \widetilde{xs}^\ast  = \widetilde{sx} $, 
 henceforth
 $ \widetilde{x} + \widetilde{xs} + \widetilde{sx} $ 
 is selfadjoint.
 
%
\begin{thm}\label{thm:2}
	If $s$ and $x $ are free non-commutative random variables such that $s $ is semicircular and $ x $ is infinitely divisible, then $ x + i[x, s] $ is infinitely divisible.
\end{thm}

	\begin{proof}
	It suffices to show that  for every $n\geq 1$ we have
		\begin{align*}
			\langle (\widehat{x}+\widehat{xs}+\widehat{sx})^n 1,1\rangle+
			\langle (\widetilde{x}+\widetilde{xs}+\widetilde{sx})^n\rangle=\kappa_n(x+i[x,s]).
		\end{align*}
		The result will follow from the fact that compound free Poisson distributions are dense in freely infinitely divisible distributions (see, for example \cite{nisp}). Note that in view of Proposition \ref{prop:41} it suffices to show that
		\begin{align*}
			&	\langle (\widehat{x}+\widehat{xs}+\widehat{sx})^n 1,1\rangle=\sum_{k=0}^{\lfloor\tfrac{n}{2}\rfloor} \sum_{\substack{i_0+\ldots+i_k=n\\i_0,i_k\geq 1\\i_1,\ldots,i_{k-1}\geq 2}} \sum_{\pi\in NC_{irr}(k+1)}\varphi_\pi (Y^{i_0},Y^{i_1},\ldots,Y^{i_k})\\
		&	\langle 
			(\widetilde{x}+\widetilde{xs}+\widetilde{sx})^n
			 1,1\rangle=\sum_{k=1}^{\lfloor\tfrac{n}{2}\rfloor} \sum_{\substack{i_1+\ldots+i_k=n\\i_1,\ldots,i_{k}\geq 2}} \sum_{\pi\in NC(k)}\varphi_\pi (Y^{i_0},Y^{i_1},\ldots,Y^{i_k}).
		\end{align*}
Indeed, the first expression
 $\langle (\widehat{x}+\widehat{xs}+\widehat{sx})^n 1,1\rangle$
 covers all configurations in the expansion of $\kappa_n(x+i[x,s])$  when $s$ are not at the first and last position of the free cumulant. The second expression corresponds to the configurations in expansion of $\kappa_n(x+i[x,s])$ in which the first and the last element in the free cumulant are $s$.
We will prove the first equality above, as the proof of the second one follows the same steps.
Denote 
\begin{align*}
	\widehat{xs}^u\left(Y_1\otimes\ldots\otimes Y_n\right)=&\begin{cases}
		Y_1\otimes\ldots\otimes Y_n\otimes Y\qquad &\mbox{if $n$ is even,}\\
		0 \qquad &\mbox{if $n$ is odd,}
	\end{cases}
\\
	\widehat{xs}^d\left(Y_1\otimes\ldots\otimes Y_n\right)=&\begin{cases}
		Y_1\otimes\ldots\otimes Y_{n-1}Y\mathbb{E}(Y_n)\qquad &\mbox{if $n$ is even,}\\
		0 \qquad &\mbox{if $n$ is odd}
	\end{cases}
\end{align*}

and

\begin{align*}
			\widehat{sx}^u\left(Y_1\otimes\ldots\otimes Y_n\right)&=\begin{cases}
		Y_1\otimes\ldots\otimes Y_n Y\otimes 1\qquad &\mbox{if $n$ is odd,}\\
		0 \qquad &\mbox{if $n$ is even.}
	\end{cases}\\
			\widehat{sx}^d\left(Y_1\otimes\ldots\otimes Y_n\right)&=\begin{cases}
		Y_1\otimes\ldots\otimes Y_{n-1}\mathbb{E}(YY_n)\qquad &\mbox{if $n>1$ and is odd,}\\
		0 \qquad &\mbox{if $n$ is even.}
	\end{cases}
\end{align*}

We have that $\widehat{sx}=\widehat{sx}^d+\widehat{sx}^u$ and that $\widehat{xs}=\widehat{xs}^d+\widehat{xs}^u$.

Observe that only products of operators in which $\widehat{xs}$ is immediately followed by $\widehat{sx}$ give non-zero contribution, as they act on different length of tensor. More precisely $\widehat{sx}$ produces an even length tensor form an odd lenght tensor, and even lenght tensors are in its kernel. Likewise $\widehat{xs}$ produces an odd lenght tensor from odd length tensors.

Next we will show that each term from the sum in Proposition \ref{prop:41} is present in the expansion of $\langle (\widehat{x}+\widehat{xs}+\widehat{sx})^n 1,1\rangle$.
 Fix $k\in\{1,\ldots\lfloor \tfrac{n}{2}\rfloor\}$ and $i_0,\ldots,i_k$ such that $i_0+\ldots+i_k=n$ where $i_0,i_k\geq 1$ and $i_1,\ldots,i_{k-1}\geq 2$. 
 And in the expansion of 
 $\langle (\widehat{x}+\widehat{xs}+\widehat{sx})^n 1,1\rangle$ 
 consider the product
\begin{align*}
	\underbrace{\widehat{x}\cdot\ldots\cdot\widehat{x} \cdot \widehat{xs}}_{i_0} \cdot \underbrace{\widehat{sx}\cdot \widehat{x}\ldots\cdot \widehat{x} \cdot \widehat{xs}}_{i_1}\cdot\ldots\cdot \underbrace{\widehat{sx}\cdot\ldots\cdot \widehat{x} }_{i_k}
\end{align*}

The formula from Proposition \ref{prop:41} says that for each such fixed groups of $x$'s we join them in a non crossing way and each block contributes moment of $Y$ of order equal to the number of $x$'s in block. Moreover since we sum only over irreducible non-crossing partitions $\pi$, then $x$'s from groups corresponding to $i_0$ and $i_k$ are always in the same block.

We have
\begin{align*}
	&\langle\underbrace{\widehat{x}\cdot\ldots\cdot \widehat{x} \cdot \widehat{xs}}_{i_0} \cdot \underbrace{\widehat{sx}\cdot \widehat{x}\ldots\cdot \widehat{x} \cdot \widehat{xs}}_{i_1}\cdot\ldots\cdot \underbrace{\widehat{sx}\cdot\ldots\cdot \widehat{x} }_{i_k} 1,1\rangle\\
	=&\langle\underbrace{\widehat{x}\cdot\ldots\cdot \widehat{x} \cdot \widehat{xs}}_{i_0} \cdot \underbrace{\widehat{sx}\cdot \widehat{x}\ldots\cdot \widehat{x} \cdot \widehat{xs}}_{i_1}\cdot\ldots\cdot \widehat{sx} \cdot Y^{i_{k}-1} ,1\rangle
\end{align*}
On length one tensors $\widehat{sx}^d$ gives zero and $\widehat{sx}^u \left(Y^{i_{k}-1}\right)=Y^{i_k}\otimes 1$. Next we apply to this either $\widehat{xs}^d$ or $\widehat{xs}^u$. We have
\begin{align*}
	\widehat{xs}^d(Y^{i_k}\otimes 1)&=Y^{i_k}\\
	\widehat{xs}^u(Y^{i_k}\otimes 1)&=Y^{i_k}\otimes 1\otimes Y
\end{align*}
Next we will apply  times the operator
 $\widehat{x}^{i_{k-1}-1}$
  to $Y^{i_k}$ and to $Y^{i_k}\otimes 1\otimes Y$. First situation corresponds to the case when in the irreducible partition $\pi$ numbers $k-1$ and $k$ are in the same block. Second situation corresponds to the case when $k-1$ is in block nested inside the block of $k$. For a non-crossing irreducible partition of the set $\{1,\ldots,k\}$ this exhausts all possibilities how numbers $k-1$ and $k$ can be placed relative to each other.

It is useful to observe that for $\pi\in NC(n)$ and $j\in\{2,\ldots,k\}$ there are four possibilities how $j-1$ and $j$ are placed relative to each other: 
\begin{itemize}
	\item[$1^o$] block containing $j$ is nested inside the block containing $j-1$,
	\item[$2^o$]  block containing $j-1$ is nested inside the block containing $j$,
	\item[$3^o$]  block containing $j-1$ and $j$ is are in the same block,
	\item[$4^o$] blocks containing $j-1$ and $j$ are adjacent to each other.
\end{itemize}
Moreover partition $\pi$ of $\{1,\ldots,k\}$ is non-crossing if and only if for all $j\in\{2,\ldots,k\}$ pair $j-1,j$ falls in one of the four cases above.

Next we will consider an odd length tensor of lenght at least three (observe that from construction, on even positions in each tensor we only have the constant random variable $1$

Consider next
\[\underbrace{\widehat{x}\cdot\ldots\cdot \widehat{x} \cdot \widehat{xs}}_{i_0}
 \cdot \underbrace{\widehat{sx}\cdot \widehat{x}\ldots\cdot \widehat{x} \cdot \widehat{xs}}_{i_1}
 \cdot\ldots\cdot \underbrace{\widehat{sx}\cdot\ldots\cdot \widehat{xs} }_{i_{j-1}}\cdot \widehat{sx} (\ldots Y_{l_{r-1}}\otimes 1 \otimes Y_{l_{r}})\]
We will analyze what we get from the application of $\widehat{xs}\cdot \widehat{sx}=(\widehat{xs}^u+\widehat{xs}^d)(\widehat{sx}^u+\widehat{sx}^d)$, the four terms we get correspond exactly to the four possible relative positions of $j-1$ and $j$ in a non--crossing partition.
\begin{itemize}
	\item Case $1^o$ above corresponds to 
	\begin{align*}
		\widehat{xs}^d\cdot \widehat{sx}^d (\ldots Y_{l_{r-1}}\otimes 1 \otimes Y_{l_{r}})&=\bE(Y_{l_{r}}Y)\widehat{xs}^d( \ldots\otimes Y_{l_{r-1}} \otimes 1)\\
		&=\bE(Y_{l_{r}}Y)( \ldots\otimes Y_{l_{r-1}}Y ).
	\end{align*}
	\item Case $2^o$ above corresponds to 
	\begin{align*}
		\widehat{xs}^u\cdot \widehat{sx}^u (\ldots Y_{l_{r-1}}\otimes 1 \otimes Y_{l_{r}})&=\widehat{xs}^u(\ldots Y_{l_{r-1}}\otimes 1 \otimes Y_{l_{r}}\otimes 1)\\
		&=\ldots Y_{l_{r-1}}\otimes 1 \otimes Y_{l_{r}}\otimes 1\otimes Y.
	\end{align*}
	\item Case $3^o$ above corresponds to 
	\begin{align*}
		\widehat{xs}^d\cdot \widehat{sx}^u (\ldots Y_{l_{r-1}}\otimes 1 \otimes Y_{l_{r}})&=\widehat{xs}^d( \ldots\otimes Y_{l_{r}}Y \otimes 1)\\
		&=\bE(Y_{l_{r}}Y)( \ldots\otimes Y_{l_{r}}Y^2 \otimes 1\otimes Y).
	\end{align*}
	\item Case $4^o$ above corresponds to 
	\begin{align*}
		\widehat{xs}^u\cdot \widehat{sx}^d (\ldots Y_{l_{r-1}}\otimes 1 \otimes Y_{l_{r}})&=\bE(Y_{l_{r}}Y)\widehat{xs}^u( \ldots\otimes Y_{l_{r-1}} \otimes 1)\\
		&=\bE(Y_{l_{r}}Y)( \ldots\otimes Y_{l_{r-1}} \otimes 1\otimes Y).
	\end{align*}
\end{itemize}

Thus products  $\widehat{xs}\cdot \widehat{sx}$ produce all possible non--crossing partitions, the operator $\widehat{x}$ only increases the power of $Y$.
 In order for the inner product \[\langle\underbrace{\widehat{x}\cdot\ldots\cdot \widehat{x} \cdot \widehat{xs}}_{i_0} \cdot \underbrace{\widehat{sx}\cdot \widehat{x}\ldots\cdot \widehat{x} \cdot \widehat{xs}}_{i_1}\cdot\ldots\cdot \underbrace{\widehat{sx}\cdot\widehat{x}\cdot\ldots\cdot \widehat{x} }_{i_k} 1,1\rangle\] 
 to be non--zero we need to finish with a length one tensor after applications of all operators, which implies that we get a power of $Y$ corresponding to sum of all $i_j$ for which $j$ is in the outer block.

One proves the second equality in exactly same way.
\end{proof}
\section*{Acknowledgement}
We would like to thank JC Wang for the very helpful and inspiring discussions, without which these results would not have been obtained, as well as for the generous support during the 2024 visit to University of Saskatchewan.

\bigskip


\bibliographystyle{abbrv}


\end{document}